\pdfoutput=1
\documentclass{amsart}
\usepackage{graphicx,color}
\usepackage{tabularx}
\usepackage{amssymb}

\theoremstyle{definition}
\newtheorem{definition}{Definition}
\newtheorem{question}{Question}

\theoremstyle{plain}
\newtheorem{theorem}{Theorem}
\newtheorem{proposition}{Proposition}
\newtheorem{fact}{Fact}
\newtheorem{corollary}{Corollary}

\theoremstyle{remark}
\newtheorem{remark}{Remark}
\newtheorem{example}{Example}

\newcommand{\tr}{
\begin{picture}(5,10)
\put(3,3){\circle{10}}
\qbezier(3,-2)(3,3)(3,8)
\qbezier(-1,-1)(3,3)(7,7)
\qbezier(7,-1)(3,3)(-1,7)
\end{picture}
}
\newcommand{\h}{
\begin{picture}(5,10)
\put(3,3){\circle{10}}
\qbezier(1,-1)(1,3)(1,8)
\qbezier(-1,3)(3,3)(8,3)
\qbezier(5,-1)(5,3)(5,8)
\end{picture}
}

\begin{document}
\title[Knot projections with reductivity two]{Knot projections with reductivity two}
\author{Noboru Ito}
\address{Waseda Institute for Advanced Study, 1-6-1 Nishi Waseda Shinjuku-ku Tokyo 169-8050, Japan}
\email{noboru@moegi.waseda.jp}
\author{Yusuke Takimura}
\address{Gakushuin Boys' Junior High School, 1-5-1 Mejiro Toshima-ku Tokyo 171-0031, Japan}
\email{Yusuke.Takimura@gakushuin.ac.jp}
\date{July 15, 2015}
\maketitle
\begin{abstract}
Reductivity of knot projections refers to the minimum number of splices of double points needed to obtain reducible knot projections.  Considering the type and method of splicing (Seifert type splice or non-Seifert type splice, recursively or simultaneously), we can obtain four reductivities containing Shimizu's reductivity, three of which are new.  In this paper, we determine knot projections with reductivity two for all four of the definitions.  We also provide easily calculated lower bounds for some reductivities.  Further, we detail properties of each reductivity, and describe relationships among the four reductivities with examples.  
\end{abstract}
\section{Introduction}\label{intro}
A {\it{knot projection}} is the image of a generic immersion of an unoriented circle into the sphere $S^2$.  By this definition, a knot projection is an unoriented generic immersed curve on $S^2$, also called a {\it{spherical curve}}.  In this study, we assume that every knot projection has at least one double point.  
As previously proven, every knot projection without $1$- and $2$-gons has $3$-gons.  
As shown in Fig.~\ref{rf0}, there are four types of $3$-gons.  
\begin{figure}[h!]
\includegraphics[width=7cm]{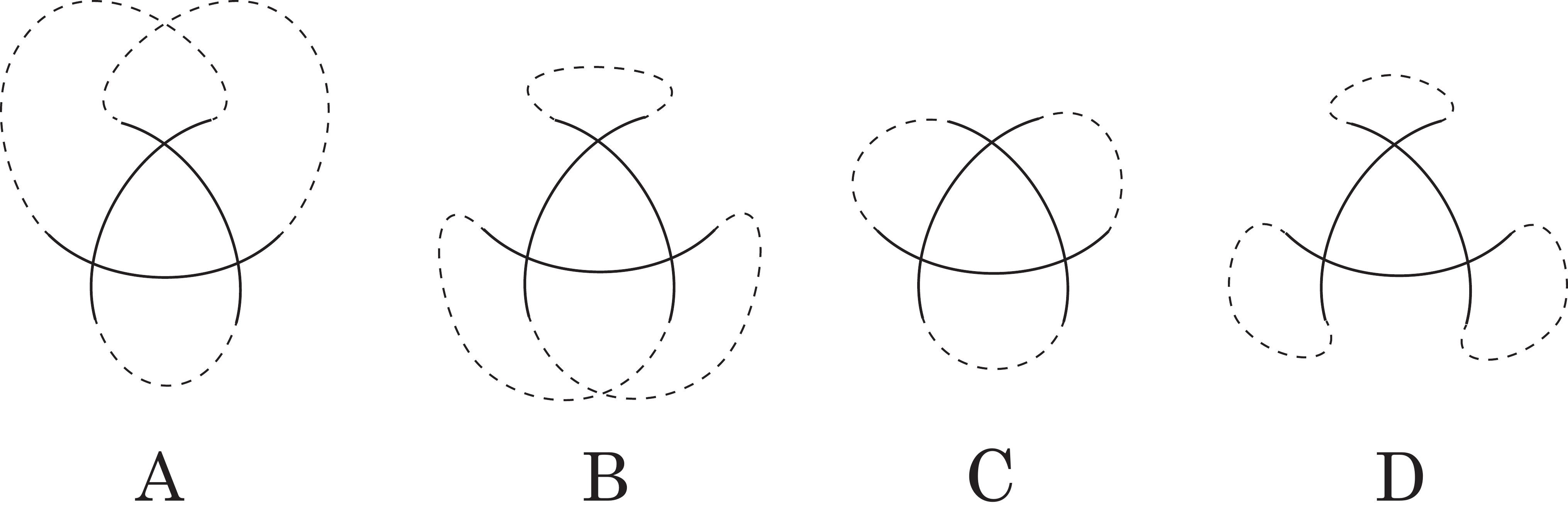}
\caption{All types of $3$-gons.  For each $3$-gon, three dotted arcs show the connections of branches of three double points.}\label{rf0}
\end{figure}
Of interest is the still-unanswered question below (\cite[Question 3.2]{S}).  
\begin{question}[Shimizu \cite{S}]\label{q1}
Is it true that every knot projection without $1$- and $2$-gons has at least one $3$-gon of type $A$, $B$, or $C$, as shown in Fig.~\ref{rf0}?  
\end{question}
If it is true, we can more easily prove the triple chord's theorem \cite{ITtriple}:
\begin{fact}[Ito-Takimura \cite{ITtriple}]
Every chord diagram of a knot projection without $1$- and $2$-gons contains \tr.    
\end{fact}
Here, a chord diagram is a circle, on which paired points are placed such that each pair of two points, corresponding to two pre-images of a double point of a knot projection, is connected by a chord.    

Shimizu's reductivity $r(P)$ \cite{S}, closely related to Question~\ref{q1}, is defined as follows: 
Shimizu's reductivity $r(P)$ of a knot projection $P$ is the minimum number of non-Seifert splices defined by Fig.~\ref{rf2}, applied recursively, to obtain a {\it{reducible knot projection}} from $P$.  
\begin{figure}[h!]
\includegraphics[width=3cm]{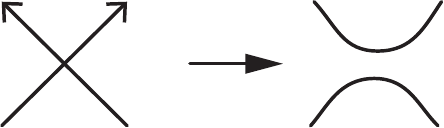}
\caption{Non-Seifert splice $A^{-1}$.}\label{rf2}
\end{figure}
Here, a reducible knot projection $P$ is a knot projection having a {\it{nugatory crossing}}, as shown in Fig.~\ref{rf1}.  A non-reducible knot projection is called a {\it{reduced knot projection}}.  Traditionally, an arbitrary non-Seifert splice is often denoted by $A^{-1}$ \cite{ItSh}.  Note that $A^{-1}$ does not depend on the selection of the orientation of a knot projection.
\begin{figure}[h!]
\includegraphics[width=3cm]{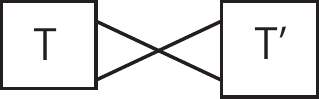}
\caption{A nugatory crossing between two (1, 1)-tangles, $T$ and $T'$.}\label{rf1}
\end{figure}
Shimizu's Question~\ref{q2} remains unanswered, as well.  
\begin{question}[Shimizu \cite{S}]\label{q2}
Is it true that $r(P) \le 3$ for every knot projection $P$?  
\end{question}
As Shimizu points out, if the answer to Question~\ref{q1} is {\it{yes}}, the answer to Question~\ref{q2} is also {\it{yes}} \cite{S}.  

In general, it is not easy to compute Shimizu's reductivity $r$.  However, this paper provides the lower bounds that can be easily calculated for an infinite family of knot projections (Theorem~\ref{thm_lower} and Examples~\ref{ex1} and \ref{ex2}).  Moreover, we explicitly obtain the necessary and sufficient condition for a knot projection to satisfy $r(P)=2$ (Theorem~\ref{r2}).  

As a first step, \cite{ITtriple} establishes the necessary and sufficient condition that a knot projection has $r(P)=1$ (Fact~\ref{f3}).  
\begin{fact}\label{f3}
Let $P$ be a reduced knot projection $P$.  There exists a circle intersecting $P$ at just two double points of $P$, as shown in Fig.~\ref{rf6}, if and only if $r(P)=1$.  
\end{fact}
As a corollary, 
\begin{fact}[Ito-Takimura \cite{ITtriple}]
Every chord diagram of a knot projection $P$ with $r(P)=1$ has a sub-chord diagram \tr.
\end{fact}
In this paper, we take this one step further, proving the following theorem and its corollary.
\begin{theorem}\label{r2}
Let $P$ be a reduced knot projection.  There is no possibility of existence of a simple closed curve, such as the one in Fig.~\ref{rf6}, and there exists a circle intersecting $P$ at just three double points of $P$, as shown in Fig.~\ref{rf5}, if and only if $r(P)=2$.  
\begin{figure}[h!]
\includegraphics[width=2cm]{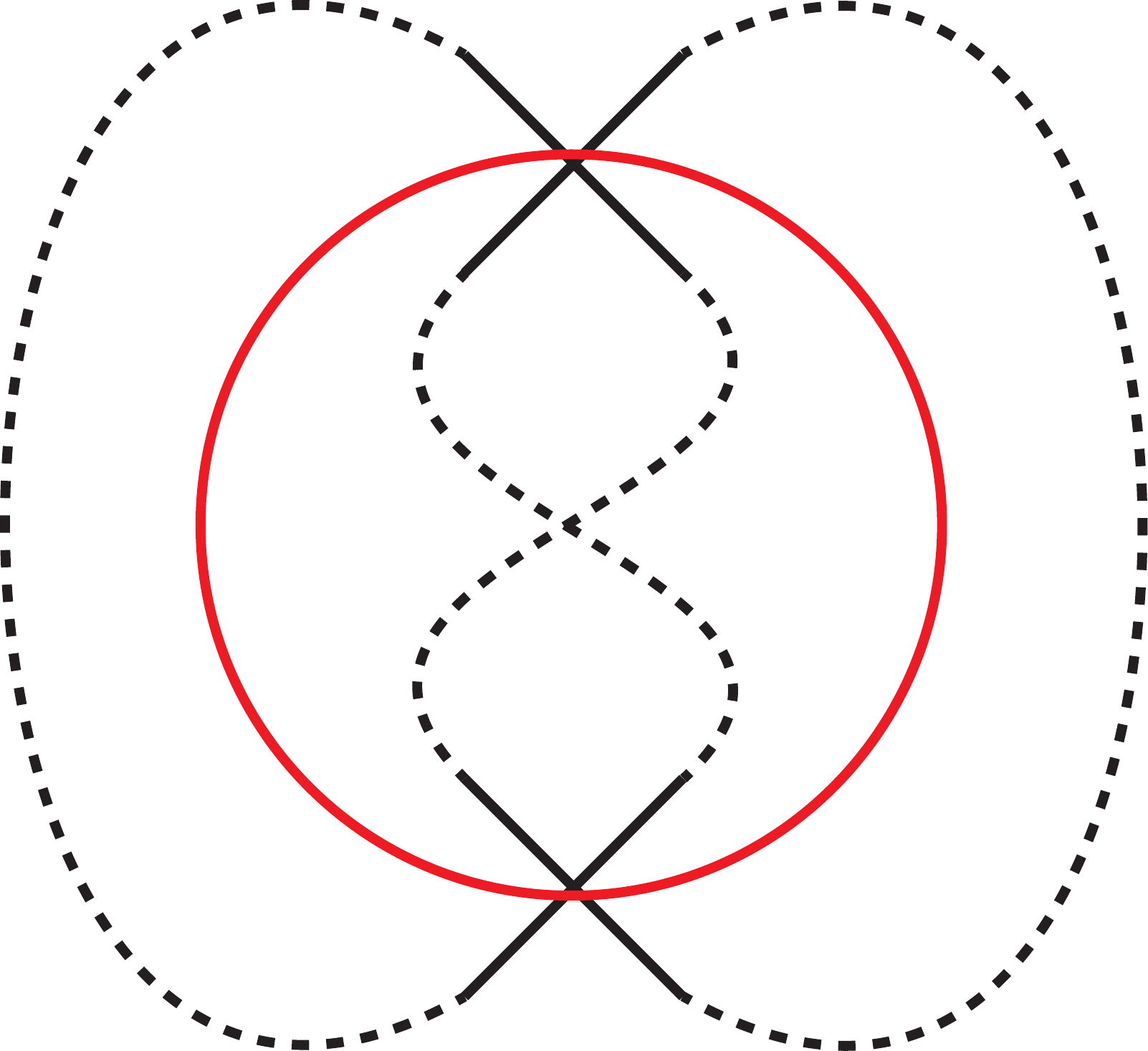}
\caption{A simple closed curve and a type of knot projection.  Dotted arcs show the connections of branches between two double points.}\label{rf6}
\end{figure} 
\begin{figure}[h!]
\includegraphics[width=10cm]{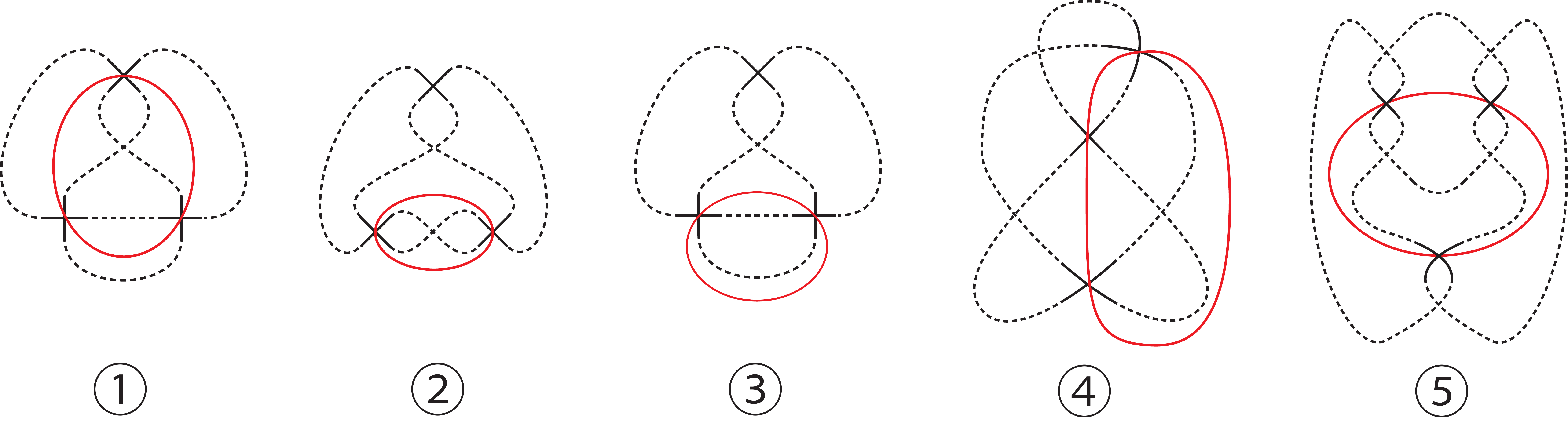}
\caption{All types of knot projections with Shimizu's reductivity two.}\label{rf5}
\end{figure}
\end{theorem}
\begin{corollary}
If $r(P)=2$ of a knot projection $P$, its chord diagram has a sub-chord diagram \h.
\end{corollary}
We also introduce three new types of reductivities, and determine knot projections up to reductivity two for each of all the four reductivities.

Here, we state that it is worth considering these three reductivities.  As for Question~\ref{q1}, there is no doubt that it is geometrically important.  We regard Shimizu's formulation (Question~\ref{q2}) as capturing an aspect of Question~\ref{q1}.  If the answer to Question~\ref{q1} is yes, then Question~\ref{q2} obtains one of the necessary conditions.  

\noindent$\bullet$ (1) If the answer to Question~\ref{q1} is yes, the answer to Question~\ref{q2} yes ($r \le 3$).  

In other words, if the answer to Question~\ref{q2} is no, disproving the conjecture for $r \le 3$ would be worth, i.e., (1)' is the contraposition of (1).  

\noindent$\bullet$ (1)'  If there exists a knot projection $P$ such that $r(P) \ge 4$, the answer to Question~\ref{q1} is no.

In this paper, we add another reductivity $t \le r$, giving another necessary condition.  The reductivity $t$ was introduced by Taniyama during the discussions with the authors (private communication).  

\noindent$\bullet$ (2) If the answer to Question~\ref{q1} is yes, $t \le 3$.  

Obviously, because (1) implies (2), (2) is easier to prove than (1) and is more accessible than Question~\ref{q2}.  Further, (2) is as important as (1) to obtain a negative answer to Question~\ref{q1}.    

To the best of our knowledge, an example with $t \ge 4$ and even $t \ge 3$, has not been found yet.  Here, we present our question (cf.~Question~\ref{q2}).

\begin{question}\label{q3}
Is it true that $t (P) \le 2$ for every knot projection $P$? 
\end{question}

Thus, it is important to explicitly obtain $P$ that satisfies $t (P) \le 2$.  Question~\ref{q3} is still open, but this study determines $P$ with $t(P) \le 2$.  To determine $t(P) \le 2$, we should determine $y(P) \le 2$ or $i(P) \le 2$, where $y$ and $i$ are two additional reductivities that are closely related to $r(P)$ and $t(P)$.  These reductivities have explicit relations with $t(P)$ (Proposition~\ref{prop1}), and thus their properties should be studied.    

This paper consists of the following sections: 
Sec.~\ref{sec_def} contains all four definitions of reductivities;
Sec.~\ref{sec_r1} describes knot projections with reductivity one for all four reductivities $t(P)$, $r(P)$, $y(P)$, and $i(P)$ for a knot projection $P$; Sec.~\ref{sec_lower} introduces easily calculated lower bounds; Sec.~\ref{sec_r2} contains a proof of Theorem~\ref{r2}; 
Sec.~\ref{sec_t2} determines knot projections with reductivity two for each of the three cases $t(P)$, $y(P)$, and $i(P)$.   Finally, Sec.~\ref{sec_table} obtains a list of reductivities for prime reduced knot projections up to seven double points.   

\section{Four definitions of reductivities}\label{sec_def}
In Definition~\ref{def_four}, reductivity $r(P)$ is introduced by Shimizu \cite{S}.  The others, $t(P)$, $y(P)$, and $i(P)$, are new.  
\begin{definition}\label{def_four}
Let $P$ be a knot projection.  
A splice indicates either a Seifert splice (Fig.~\ref{rf4}) or a non-Seifert splice (Fig.~\ref{rf2}).  
\begin{itemize}
\item $t(P)$ $=$ min $\{$ number of splices, applied simultaneously, to obtain a reducible knot projection from $P$ $\}$,
\item $r(P)$ $=$ min $\{$ number of non-Seifert splices, applied recursively, to obtain a reducible knot projection from $P$ $\}$ (\cite{S}),
\item $y(P)$ $=$ min $\{$ number of non-Seifert splices, applied simultaneously, to obtain a reducible knot projection from $P$ $\}$,
\item $i(P)$ $=$ min $\{$ number of Seifert splices, applied simultaneously, to obtain a reducible knot projection from $P$ $\}$.
\end{itemize}
\begin{figure}[h!]
\includegraphics[width=3cm]{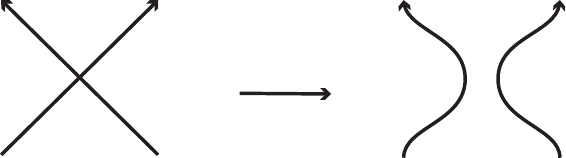}
\caption{A Seifert splice.  The splice preserves the orientation of a curve.}\label{rf4}
\end{figure}
\end{definition}
By the definition, we have the basic properties given in Proposition~\ref{prop1} below.  Note that the coherent $2$-gon referenced in the properties is $2$-gon shown in Fig.~\ref{coh}.
\begin{figure}[h!]
\includegraphics[width=2cm]{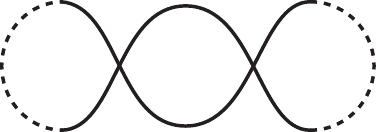}
\caption{A coherent $2$-gon.  Dotted arcs show the connections of branches of two double points of $2$-gons.}\label{coh}
\end{figure}  
\begin{proposition}\label{prop1}
Let $P$ be a reduced knot projection.  Then $P$ has the following properties.  
\begin{enumerate}
\item $t(P) \le r(P)$.
\item $t(P) \le y(P)$.
\item $t(P) \le i(P)$.
\item $i(P)$ is an even nonnegative integer.
\item If a reduced knot projection $P$ has at least one coherent $2$-gon, $i(P)=2$. 
\item If a reduced knot projection $P$ has at least one $C$-type $3$-gon, $i(P)=2$.
\item If a reduced knot projection $P$ has at least one $A$-type $3$-gon, $1 \le y(P) \le 2$.
\item For any even positive integer $2m$, there exists $P$ such that $i(P)=2m$. 
\end{enumerate}
\end{proposition}
\begin{proof}
It is sufficient to verify properties (4) and (8).  

For (4), if a single Seifert splice is applied to a knot projection, the result obtained is not a knot projection.   Property (8) can be verified if we consider $(2, 2m+1)$-torus knot projections.  Hence, the proposition is verified.  
\end{proof}
\begin{example}\label{ex1}
There is an example such that $t(P) \lneqq r(P)$.
\begin{figure}[h!]
\includegraphics[width=3cm]{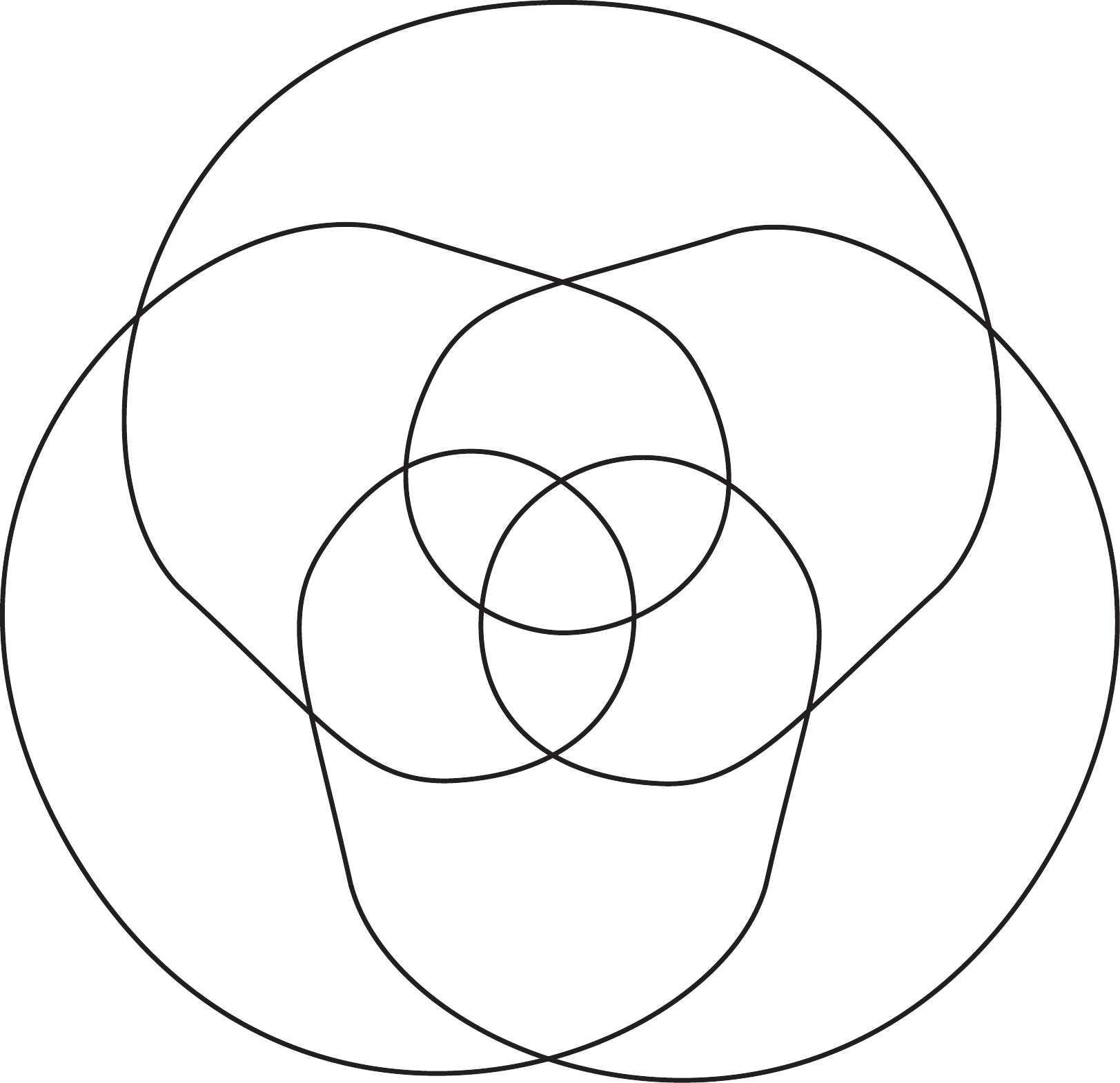}
\caption{Example with $t(P)=2$, $r(P)=3$, $y(P)=2$, and $i(P)=2$.}\label{exa1}
\end{figure}
\end{example}
\begin{proof}
\begin{figure}[h!]
\includegraphics[width=10cm]{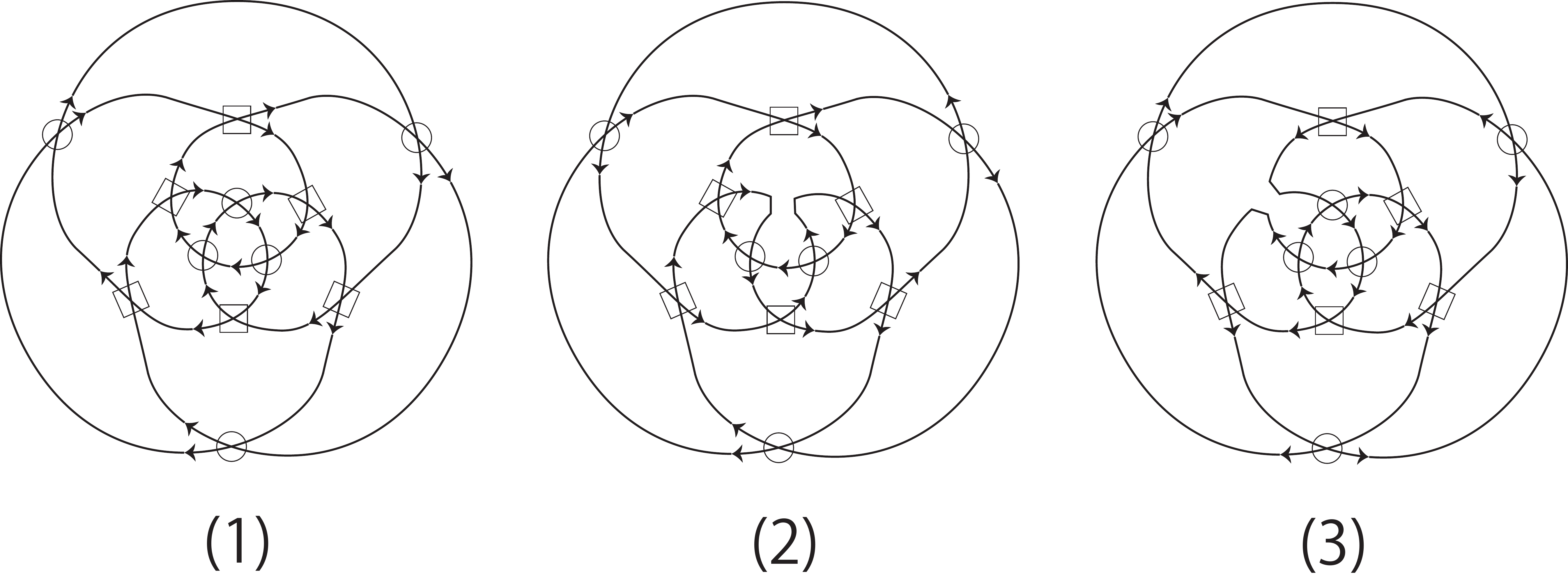}
\caption{Knot projections (1), (2) and (3).}\label{expr}
\end{figure}
From Fig.~\ref{expr} (1), it can be seen that due to its symmetry, $P$ has exactly two types of double points.  Thus, it can be easily seen that any single splice cannot produce a reducible knot projection.  Thus, $2 \le t(P), r(P), y(P),$ and $i(P)$. (More simply, because $|\tau(P)|=1$, we have $2 \le t(P), r(P), y(P), i(P)$ (Theorem~\ref{thm_lower}).)   From Fig.~\ref{expr} (2) and (3), we can easily notice that $t(P), y(P),$ and $i(P) = 2$.  However, from Fig.~\ref{expr} (2) and (3), we can see that $r(P) = 3$.  
\end{proof}
Intuitively, as in this case, we can see that the calculations for $t(P)$, $y(P)$, and $i(P)$ are simpler than those for $r(P)$.  This is one of the advantages of considering the reductivities $t(P)$, $y(P)$, and $i(P)$.  Similarly to Example~\ref{ex1}, we have Example~\ref{ex2}.  
\begin{example}\label{ex2}
There are examples such that $t(P) \lneqq y(P)$ and $r(P) \lneqq y(P)$.
\begin{figure}[h!]
\includegraphics[width=8cm]{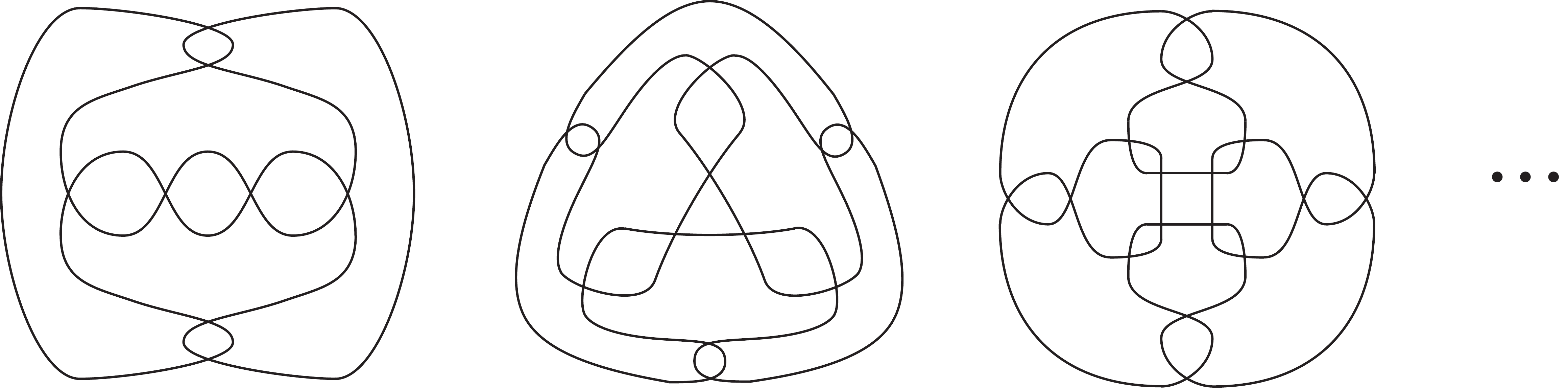}
\caption{Examples (infinite family) with $t(P)=2$, $r(P)=2$, and $y(P) \gneqq 2$ ($i(P)=2$).}\label{exa2}
\end{figure}
\end{example}
\section{Reductivity one}\label{sec_r1}
\begin{theorem}\label{r1}
For every reduced knot projection $P$, there exists a circle intersecting $P$ at just two double points of $P$, as shown in Fig.~\ref{rf6}, if and only if $t(P)=1$.  
\end{theorem}
\begin{proof}
\begin{itemize}
\item (If part) Consider the case $t(P)=1$.  Let $P'$ be a reduced knot projection obtained from $P$ by applying any splice at a double point, say $a$, of $P$, and let $b$ be a reducible crossing of $P'$.  Then, there exists a simple closed curve that intersects $P$ with $a$ and $b$ only.  There are exactly four cases with respect to the connectivity among the arcs at $a$ and $b$ as shown in Fig.~\ref{rf7}.  
\begin{figure}[h!]
\includegraphics[width=8cm]{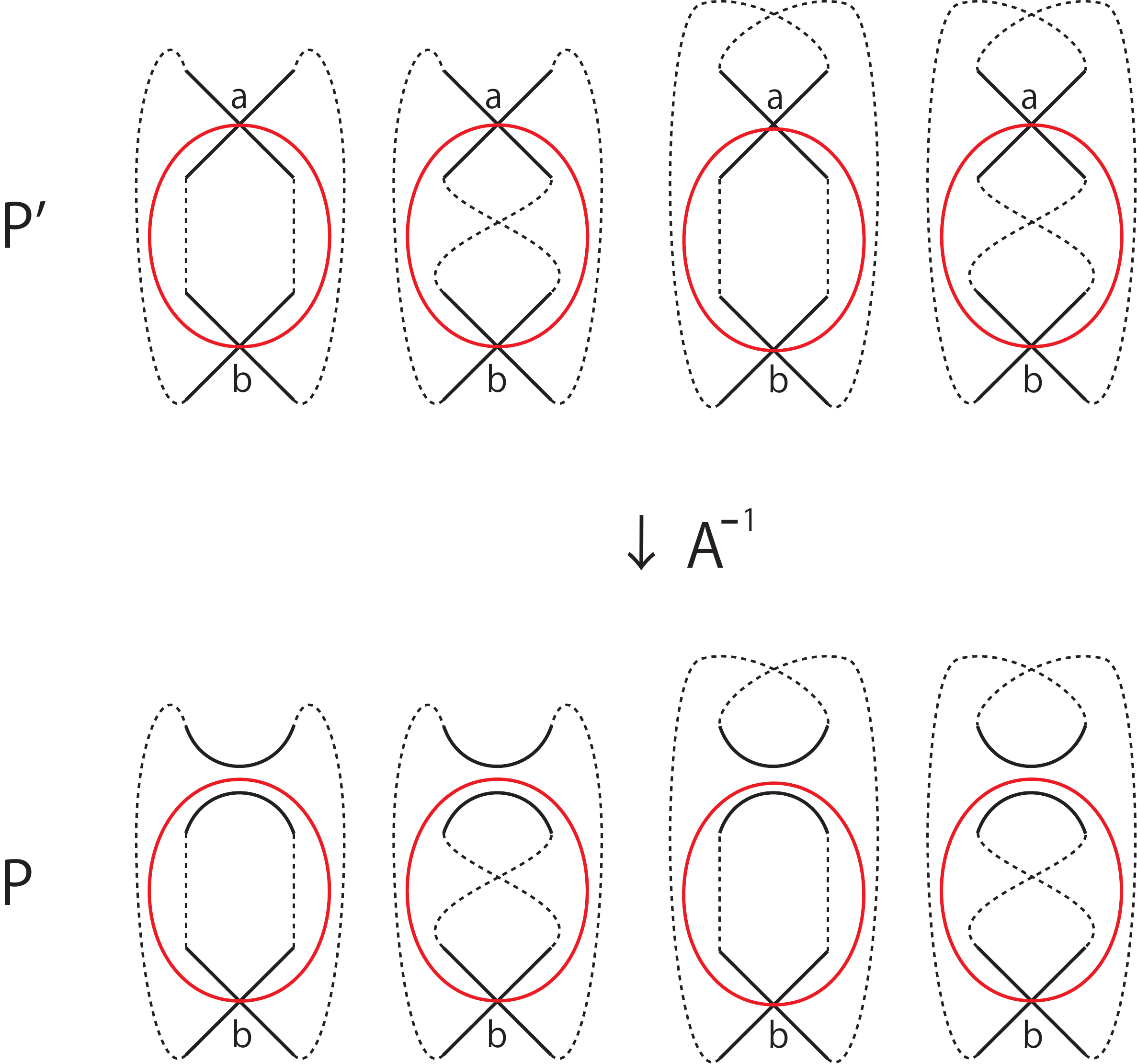}
\caption{All the possibilities for the distributions of arcs and the circle splitting the sphere into two disks before (upper) and after (lower) applying $A^{-1}$.}\label{rf7}
\end{figure}
Since $P$ is one component, the second and third cases remain.  However, the second and third cases are the same on $S^2$.    
\item (Only if part) Recall Fact~\ref{f3}.  By the definition of $t(P)$, we immediately observe that $t(P) \le r(P)$ (Proposition~\ref{prop1}).  If there exists a circle with two double points, as shown in Fig.~\ref{rf6}, this reduced knot projection $P$ satisfies $1 \le t(P) \le r(P)=1$.  Therefore, $t(P)=1$.  
\end{itemize}
\end{proof}
\begin{theorem}
For every reduced knot projection $P$, there exists a circle intersecting $P$ at just two double points of $P$, as shown in Fig.~\ref{rf6}, if and only if $y(P)=1$.  
\end{theorem}
\begin{proof}
\begin{itemize}
\item (If part) Consider a reduced knot projection $P$ with $y(P)=1$.  By the definition, $t(P) \le y(P)$ and $P$ satisfies $t(P) = 1$.  By Theorem~\ref{r1}, for $P$, there exists a circle with two double points, as shown in Fig.~\ref{rf6}.  
\item (Only if part) Obvious.  
\end{itemize}
\end{proof}
\begin{corollary}\label{try_cor}
Let $P$ be a knot projection.   The following conditions are mutually equivalent.   

\noindent $(1)$ $t(P)=1$.

\noindent $(2)$ $r(P)=1$.

\noindent $(3)$ $y(P)=1$.

\end{corollary}
\begin{corollary}
Let $P$ be a reduced knot projection.  

\noindent $(1)$ If $r(P)=2$, then $t(P)=2$.  

\noindent $(2)$ If $y(P)=2$, then $t(P)=2$.
\end{corollary}
\begin{remark}
There is no knot projection with $i(P)=1$ (see Proposition~\ref{prop1}~(4)).
\end{remark}
\section{Lower bounds for reductivities.}\label{sec_lower}
In this short section we show that {\it{circle numbers}}, first introduced in \cite{ITcircle}, are useful to obtain a lower bound of $r(P)$, $t(P)$, and $y(P)$.  
\begin{definition}[circle number \cite{ITcircle}]
A circle number $|\tau(P)|$ of a knot projection $P$ is the number of circles that result from applying splice $A^{-1}$ to every double point of $P$ simultaneously.  
\end{definition}
\begin{theorem}\label{thm_lower}
Let $P$ be a reduced knot projection.  If $|\tau(P)|=1$, then
\[2 \le t(P), 2 \le r(P),~{\text{and}}~2 \le y(P).\]
\end{theorem}
\begin{proof}
From Proposition~\ref{prop1} (1) and (2), it is clear that it is sufficient to show that $2 \le t(P)$ if $|\tau(P)|=1$.  Fig.~\ref{rf6} implies that $2 \le |\tau(P)|$ if $t(P)=1$.
\end{proof}
\begin{example}
See Figs.~\ref{exa1} and \ref{exa2}.
\end{example}
\begin{remark}
Known upper bounds are shown in Table \ref{table0} (see \cite{S} or Proposition~\ref{prop1}).  For example, if a reduced knot projection $P$ has at least one coherent $2$-gon, $t(P) \le 2$. 
\begin{table}[h!]
\includegraphics[width=5cm]{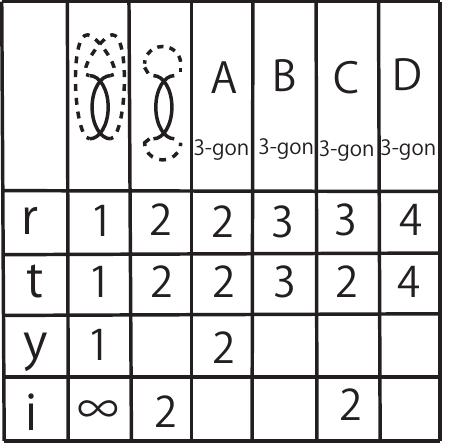}
\caption{Known upper bounds.  Blank cells denote unknown information.}\label{table0}
\end{table}
\end{remark}
\section{Shimizu's reductivity two}\label{sec_r2}
In this section, we prove Theorem~\ref{r2}.  The inverse operation $A$ of the non-Seifert splice $A^{-1}$ is the operation shown in Fig.~\ref{rf9}.  

\noindent{\it{Proof of Theorem~\ref{r2}.}}
\begin{figure}[h!]
\includegraphics[width=5cm]{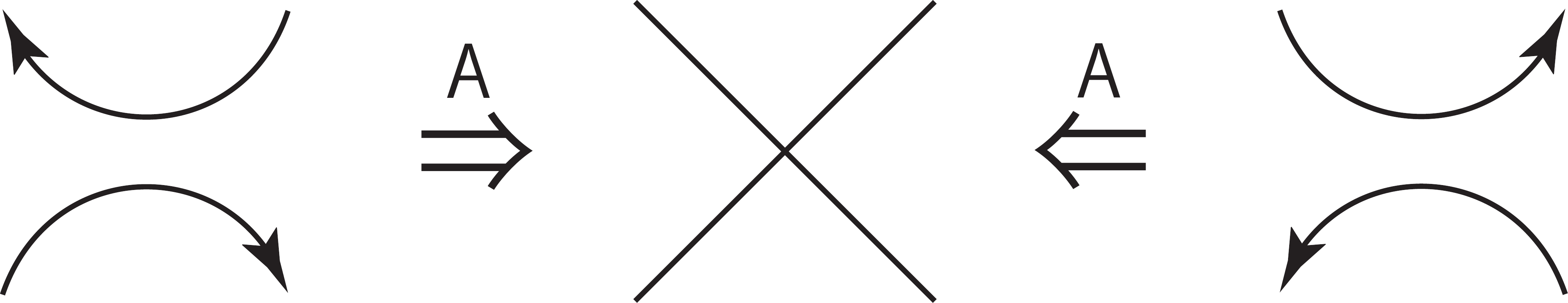}
\caption{Inverse replacements of a non-Seifert splice $A^{-1}$ when an arbitrary orientation is obtained.}\label{rf9}
\end{figure}
Applying an $A$ to the knot projection in Fig.~\ref{rf8} on the dotted two arcs produces the following three cases.   (Note that in Fig.~\ref{rf10}, we chose an orientation of a knot projection, but another choice of the orientation implies the same conclusion for unoriented knot projections.)  
\begin{figure}[h!]
\includegraphics[width=2cm]{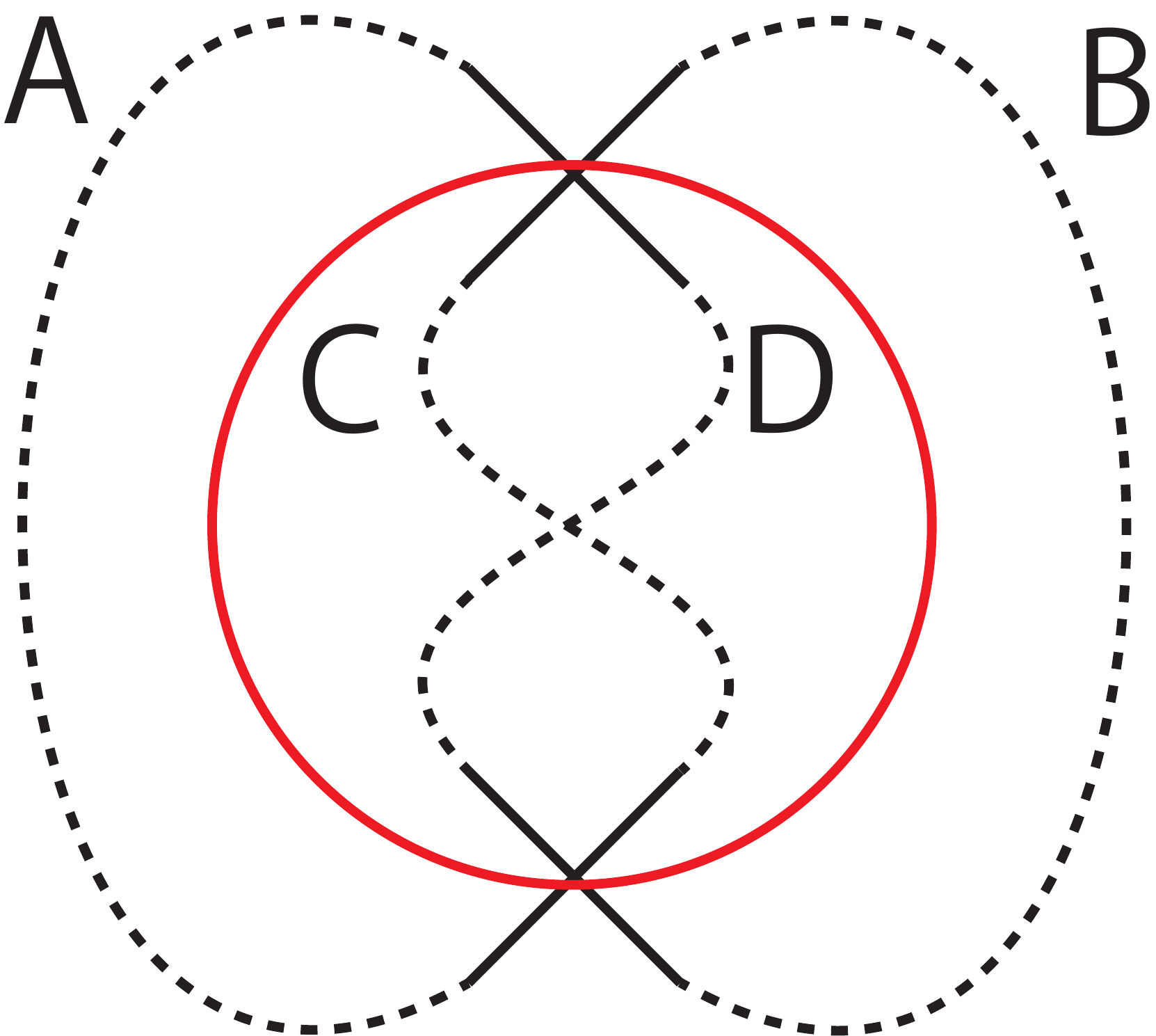}
\caption{Dotted arcs $A$, $B$, $C$, and $D$.}\label{rf8}
\end{figure}
\begin{figure}[h!]
\begin{tabular}{|c|c|c|c|} \hline
Case 1 & Case 2-(i) & Case 2-(ii) & Case 3 \\ \hline
\includegraphics[width=2.5cm]{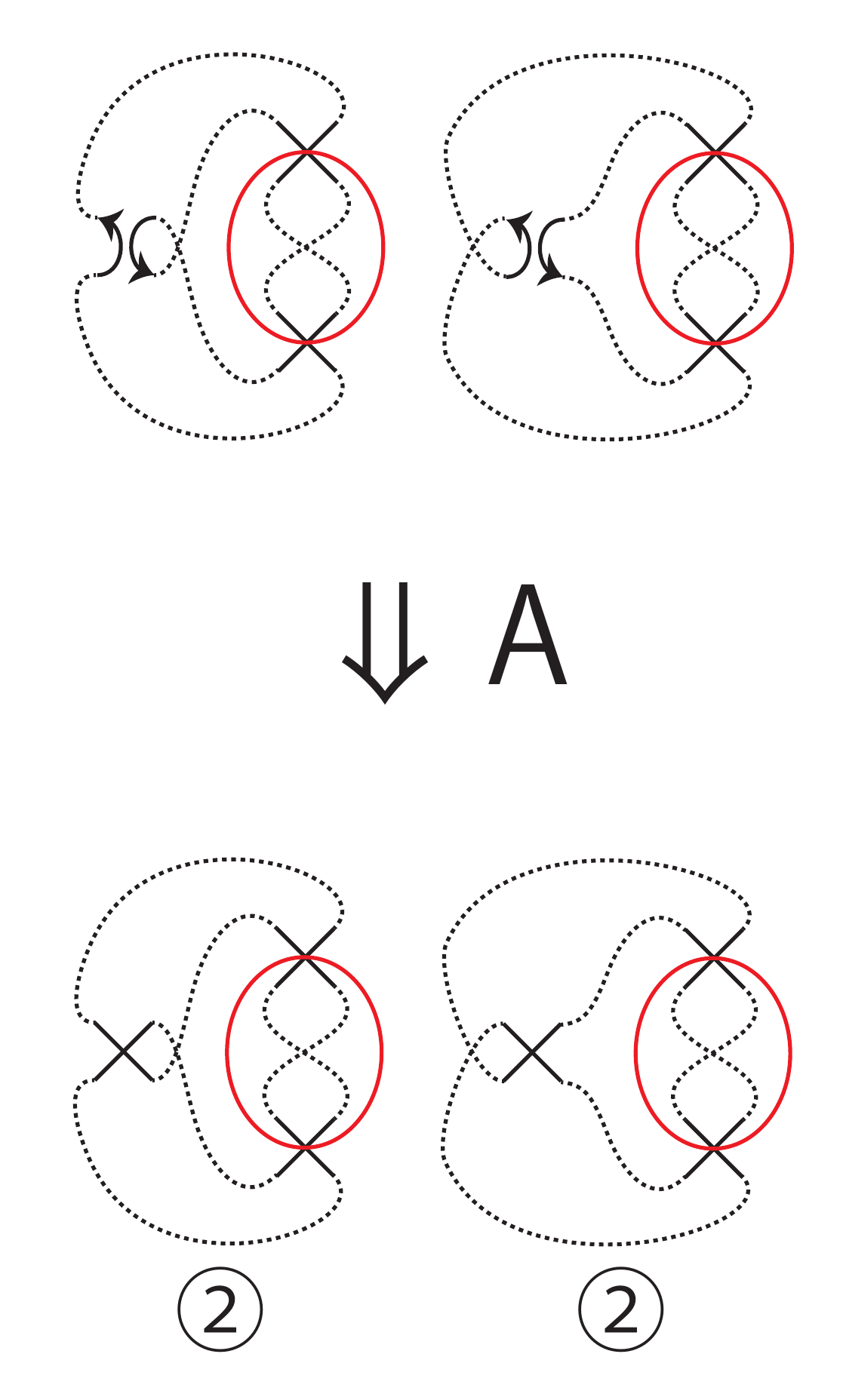}&
\includegraphics[width=2.5cm]{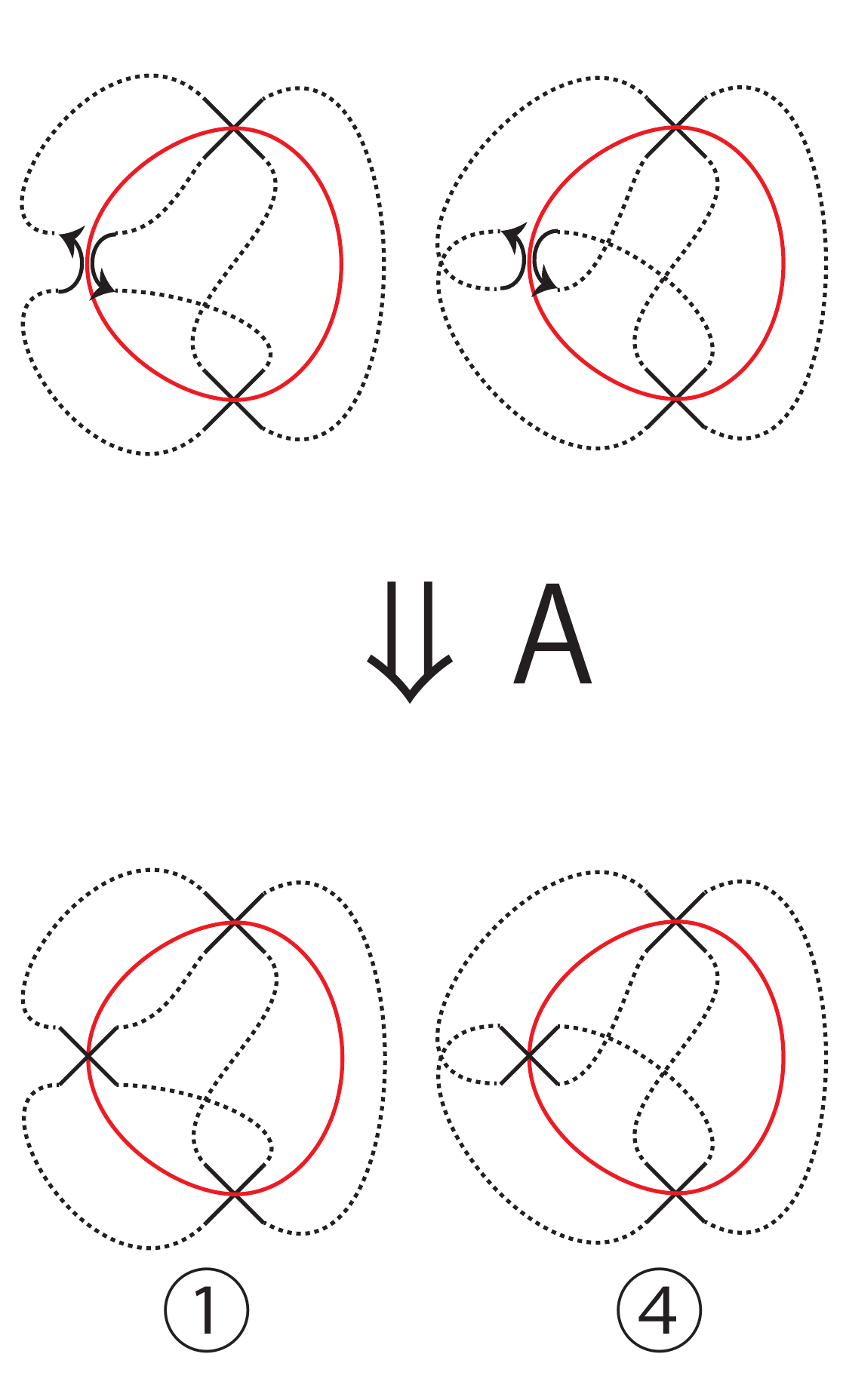}&
\includegraphics[width=2.5cm]{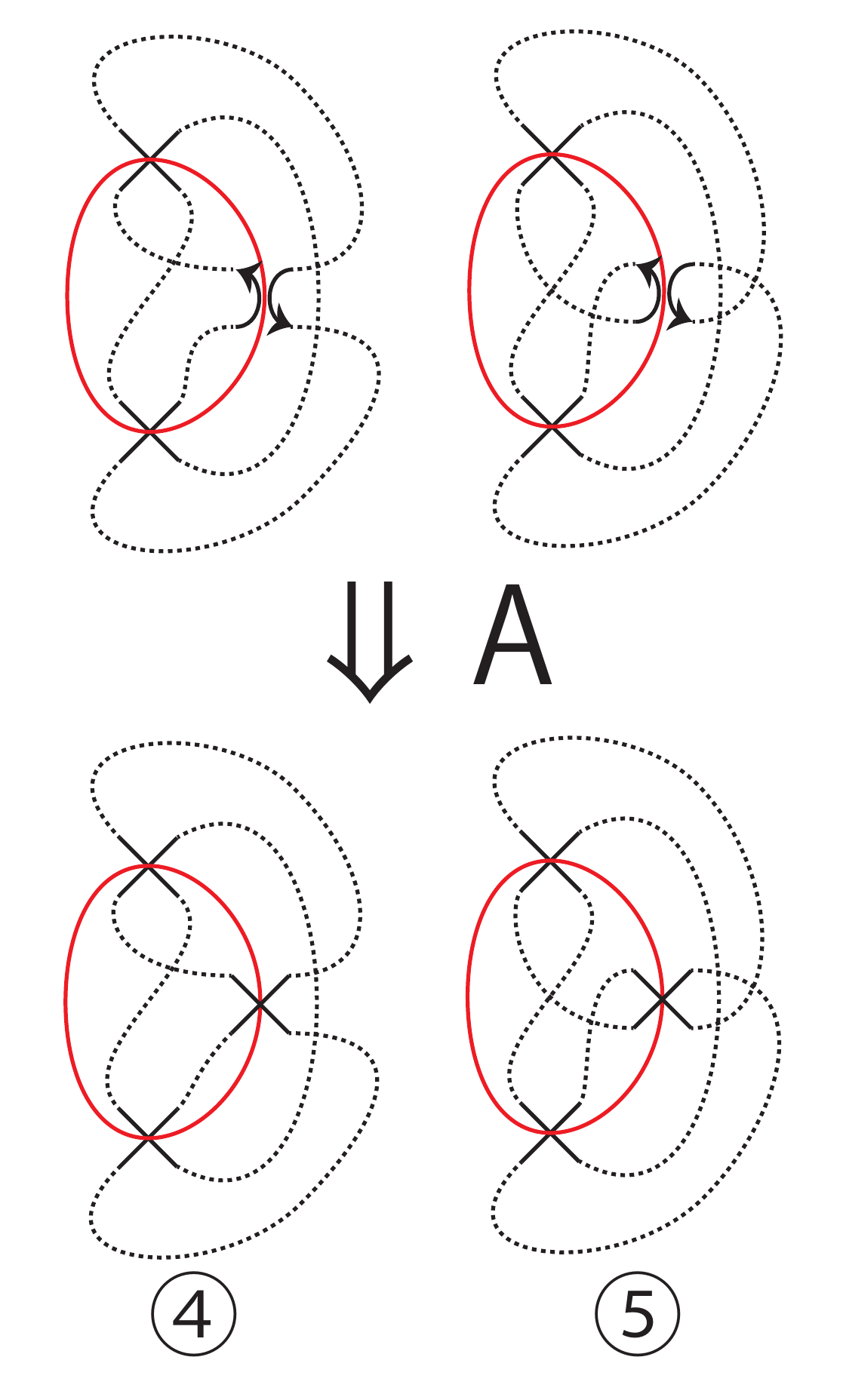}&
\includegraphics[width=2.5cm]{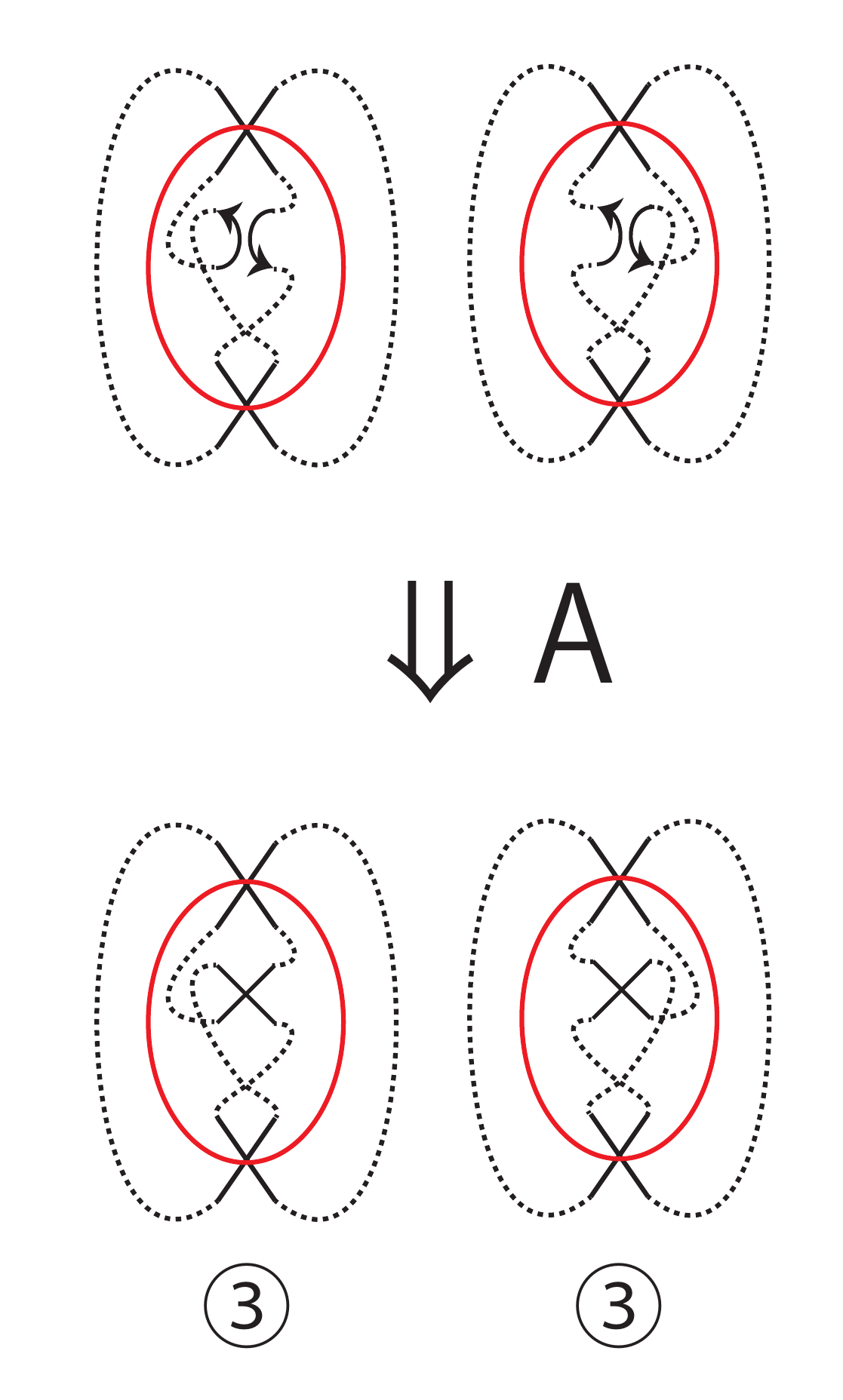}\\ \hline
\end{tabular}
\caption{Cases 1--3.}\label{rf10}
\end{figure}
\begin{itemize}
\item Case 1: Connect dotted arcs A and B. In this case, the second numbered knot projections are obtained (Fig.~\ref{rf10}).
\item Case 2: Connect dotted arcs A and C.  Pairs (A, D), (B, C), and (B, D) return the case (A, C).  (Note that each mirror image of every knot projection in Fig.~\ref{rf5} is the same as itself).  In this case, the the first, fourth, or fifth numbered knot projections are obtained (Fig.~\ref{rf10}).  
\item Case 3: Connect dotted arcs C and D.  In this case, the third numbered knot projections are obtained (Fig.~\ref{rf10}). \end{itemize}
\hfill$\Box$
\section{Other reductivities two}\label{sec_t2}
In this section, we determine knot projections with $i(P)$, $y(P)$, or $t(P)$ $\le 2$.  The inverse operation of a non-Seifert splice (resp.~Seifert splice) is denoted by $A$ (resp.~$B$).  
When proving claims in this section, we used an orientation of every knot projection for presenting an unoriented knot projection.  This orientation eliminates confusion.  Since a knot projection is a one-component curve, it is easy to find an unoriented knot projection by simply omitting the orientation.  Thus, references to the orientations of figures may be unnecessary in the proofs of Theorems~\ref{i_kiyakudo2_thm}, \ref{y_kiyakudo2_thm}, and \ref{t_kiyakudo2_thm}.  
Note that we do not distinguish between a knot projection and its mirror image in this section.  

Based on Fig.~\ref{rf11}, we consider all possible connections of two local places.
\begin{figure}[h!]
\includegraphics[width=5cm]{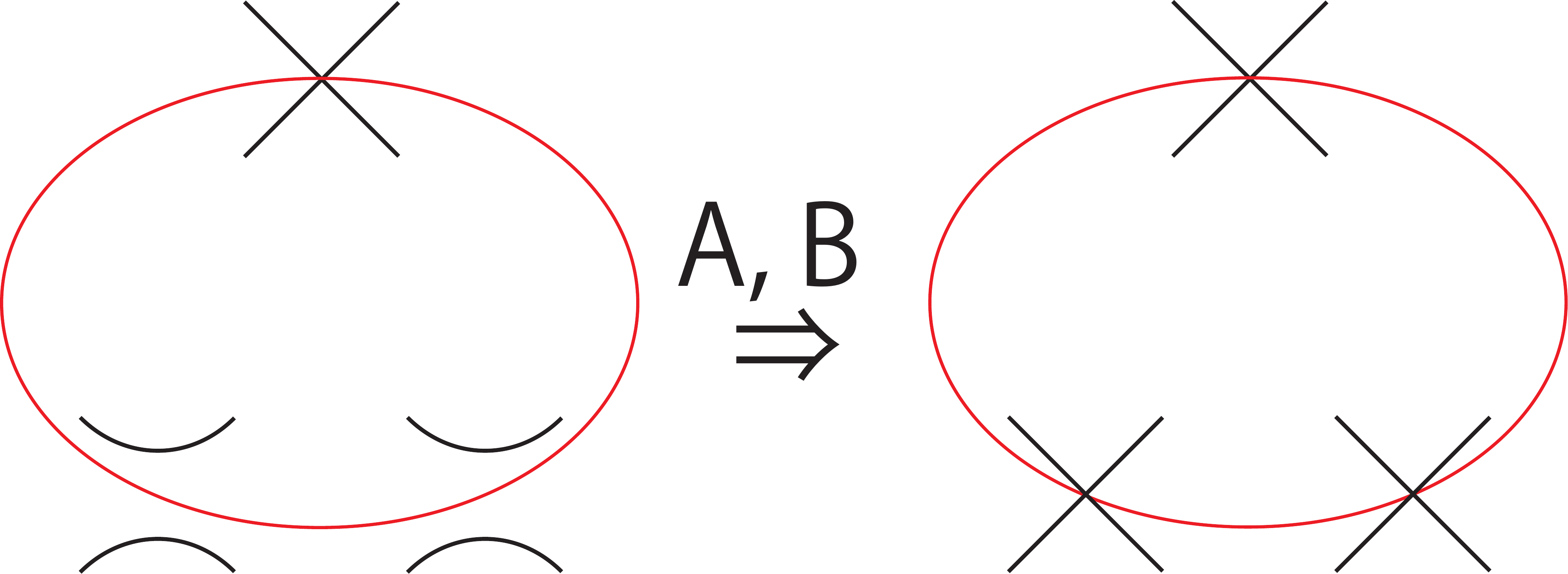}
\caption{Inverse moves $A$ or $B$ to consider all possibilities of producing a reducible knot projection after applying Seifert splice $A^{-1}$ or non-Seifert splice $B^{-1}$.}\label{rf11}
\end{figure}
When we choose $B$ from these possibilities, and we consider Fig.~\ref{rf12b}, we have Theorem~\ref{i_kiyakudo2_thm}.   
\begin{theorem}\label{i_kiyakudo2_thm}
Let $P$ be a reduced knot projection.  There exists a circle intersecting $P$ at just two or three double points of $P$, as shown in Fig.~\ref{ito}, if and only if $i(P)=2$.  
\begin{figure}[h!]
\includegraphics[width=8cm]{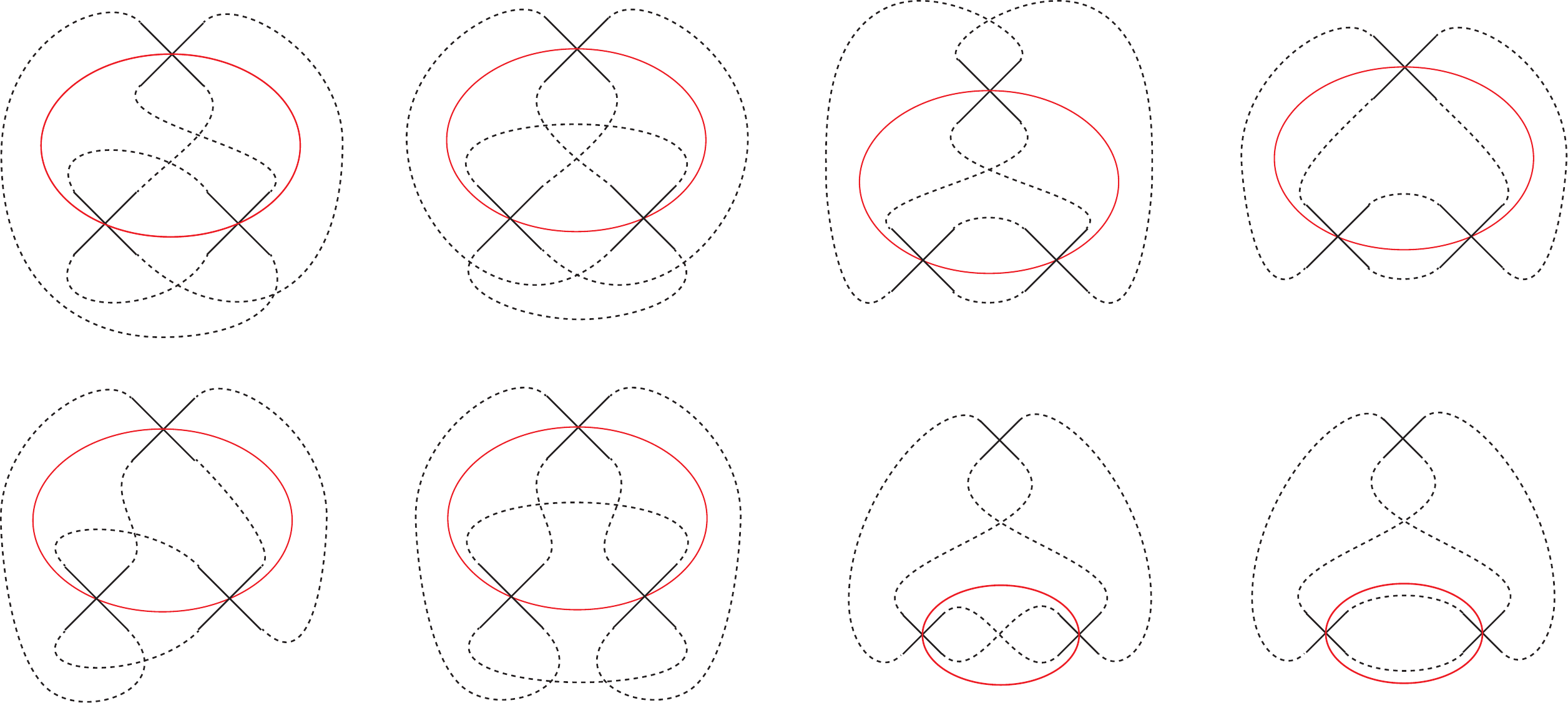}
\caption{Knot projections with $i(P) = 2$.  
}\label{ito}
\end{figure}
\end{theorem}
\begin{figure}
\includegraphics[width=10cm]{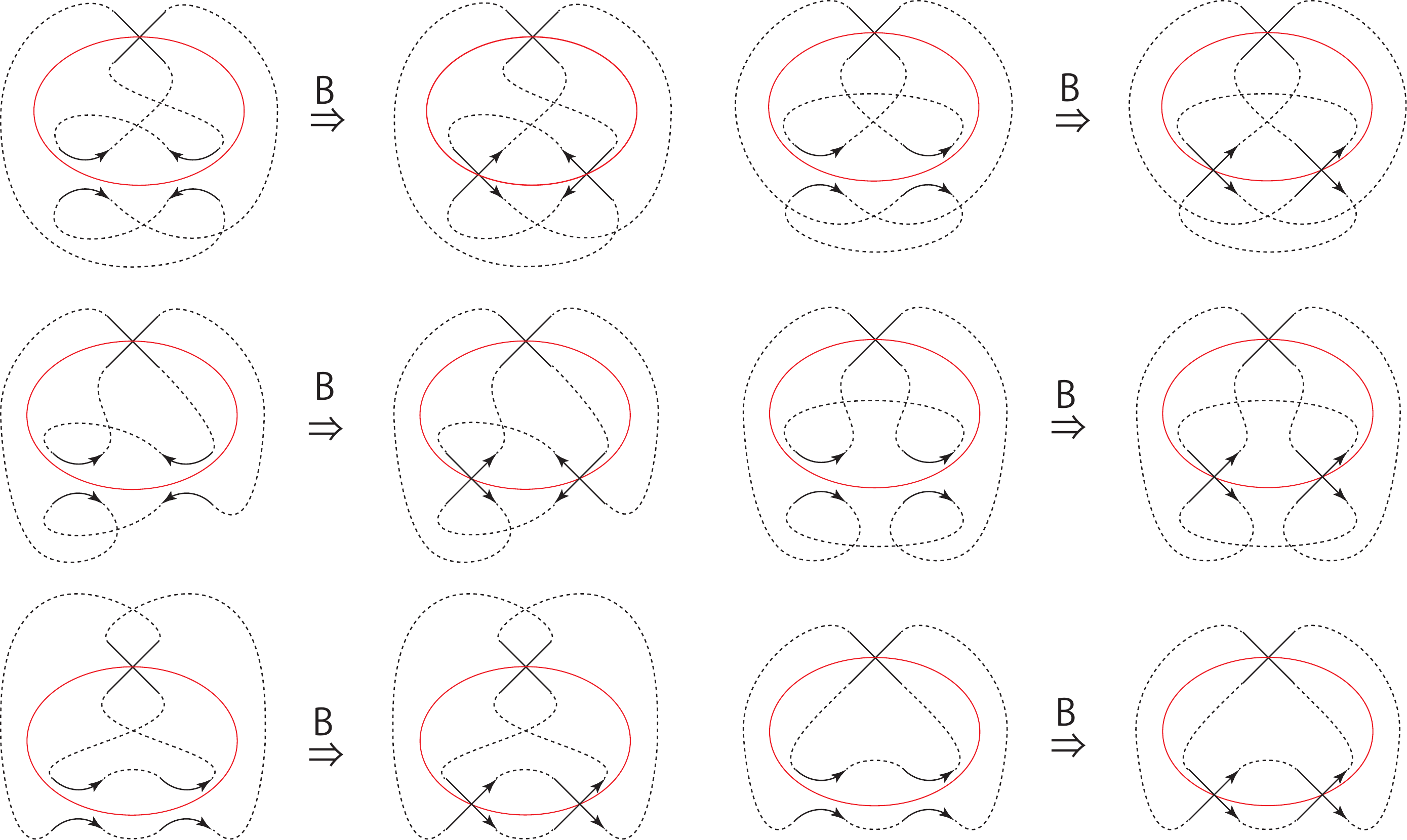}
\caption{Knot projections with $i(P)=2$.}\label{rf12}
\end{figure}
\begin{figure}[h!]
\includegraphics[width=5cm]{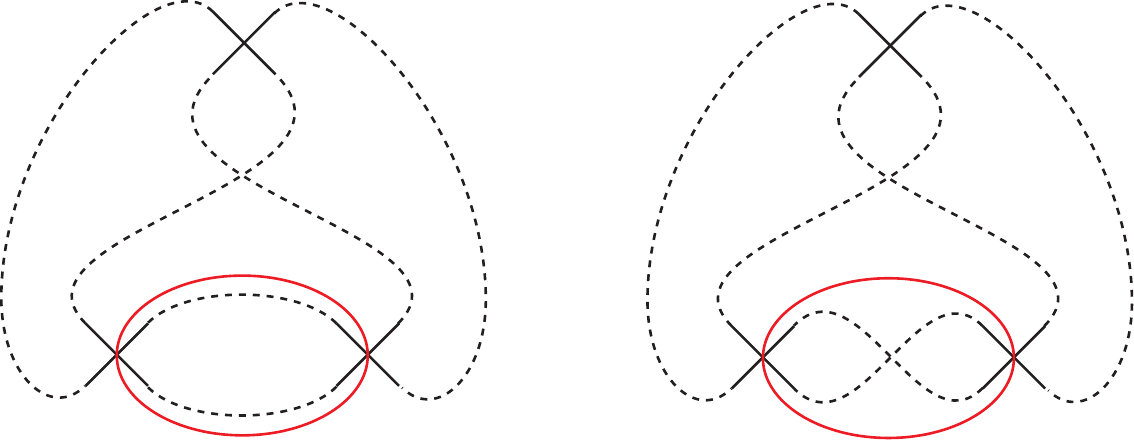}
\caption{Knot projections with $i(P)=2$.
}\label{rf12b}
\end{figure}
\begin{proof}
One Seifert splice changes the number of components of a spherical curve, and, therefore, $2 \le i(P)$ for a reduced knot projection $P$.  As shown in Fig.~\ref{rf12} and Fig.~\ref{rf12b}, we obtain all possible ways to apply $B$ at the two places of a reducible knot projection.  
\end{proof}

Next, we consider the necessary and sufficient condition $y(P)=2$.  Similar to the above, we consider any possibilities of connections at two places, as shown in Fig.~\ref{rf11}.  It is clear that the case in the first line of Fig.~\ref{rf13} is impossible.  All the possible cases appear in the second line of Fig.~\ref{rf13}. 
\begin{figure}[htbp]
\includegraphics[width=12cm]{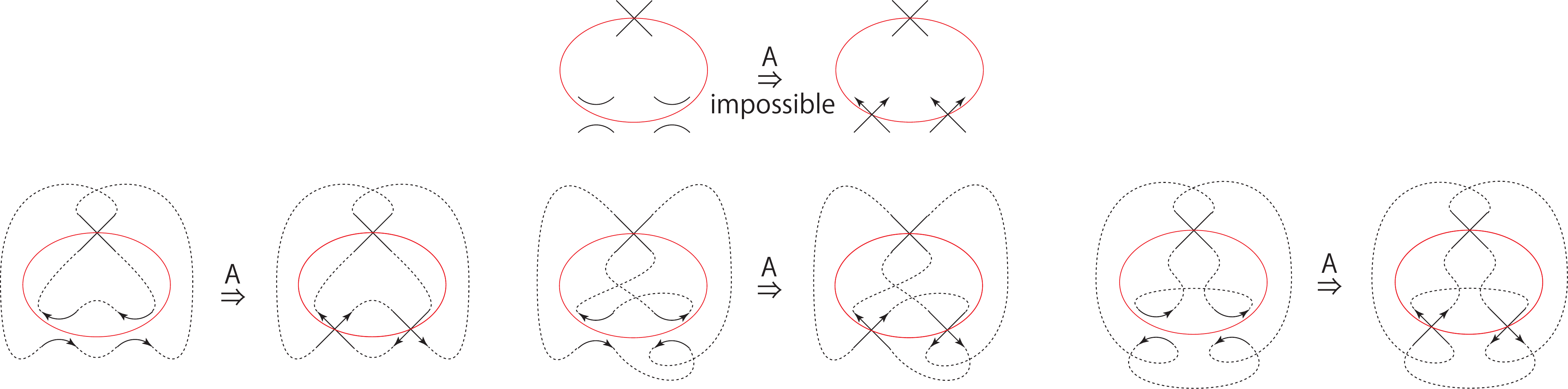}
\caption{Knot projections with $y(P) \le 2$.}\label{rf13}
\end{figure} 
\begin{theorem}\label{y_kiyakudo2_thm}
Let $P$ be a knot projection.  There is no possibility of existence of a simple circle, as in Fig.~\ref{rf6}, and there exits a circle intersecting $P$ at just three double points of $P$, as shown in Fig.~\ref{yusuke}, 
if and only if $y(P)=2$.   
\begin{figure}[h!]
\includegraphics[width=8cm]{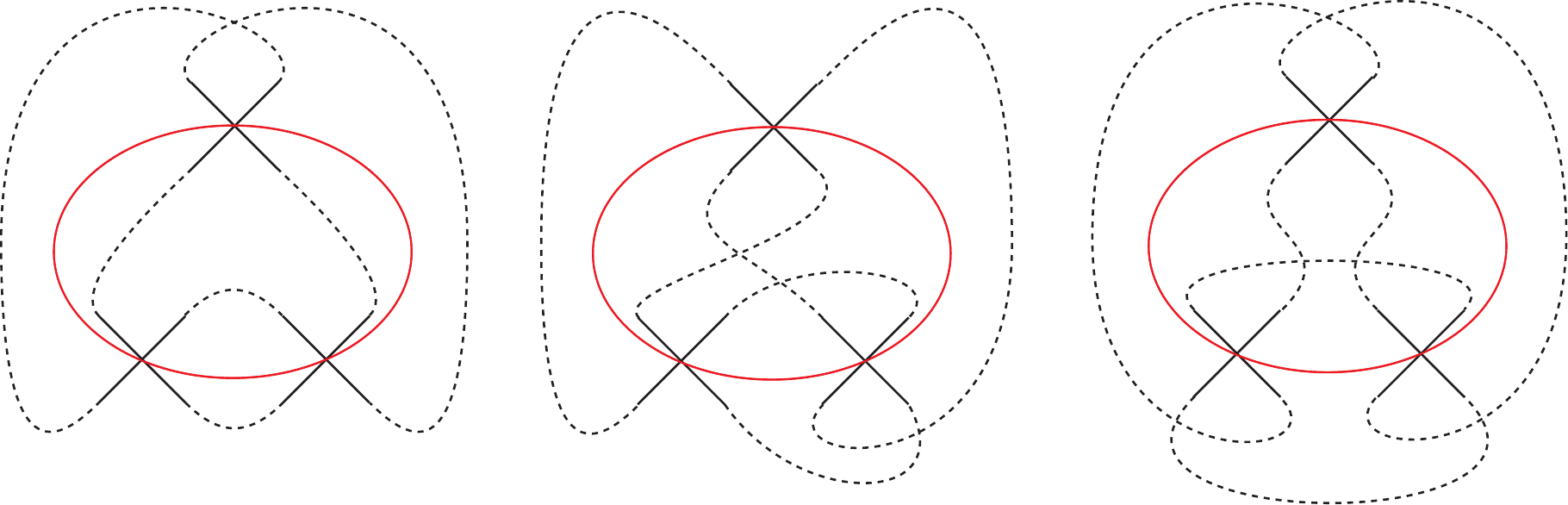}
\caption{Knot projections with $y(P) \le 2$.}\label{yusuke}
\end{figure}
\end{theorem}
\begin{proof}
Fig.~\ref{rf13} shows all possible ways to apply $A$ at the two places of a reducible knot projection.  
\end{proof}
\begin{corollary}
If $y(P)=2$ for a knot projection $P$, its chord diagram has a sub-chord diagram \h.  
\end{corollary}

Finally, we determine knot projections with $t(P)=2$ using Theorems~\ref{i_kiyakudo2_thm} and \ref{y_kiyakudo2_thm}.  Based on Fig.~\ref{rf11}, we have considered all possibilities of connections consisting of one kind of operation, either (B, B) or (A, A), at two places.  Therefore, we add all the cases of mixed type (A, B).  There are four such cases. (See Fig.~\ref{rf14}.)  
\begin{figure}[h!]
\includegraphics[width=10cm]{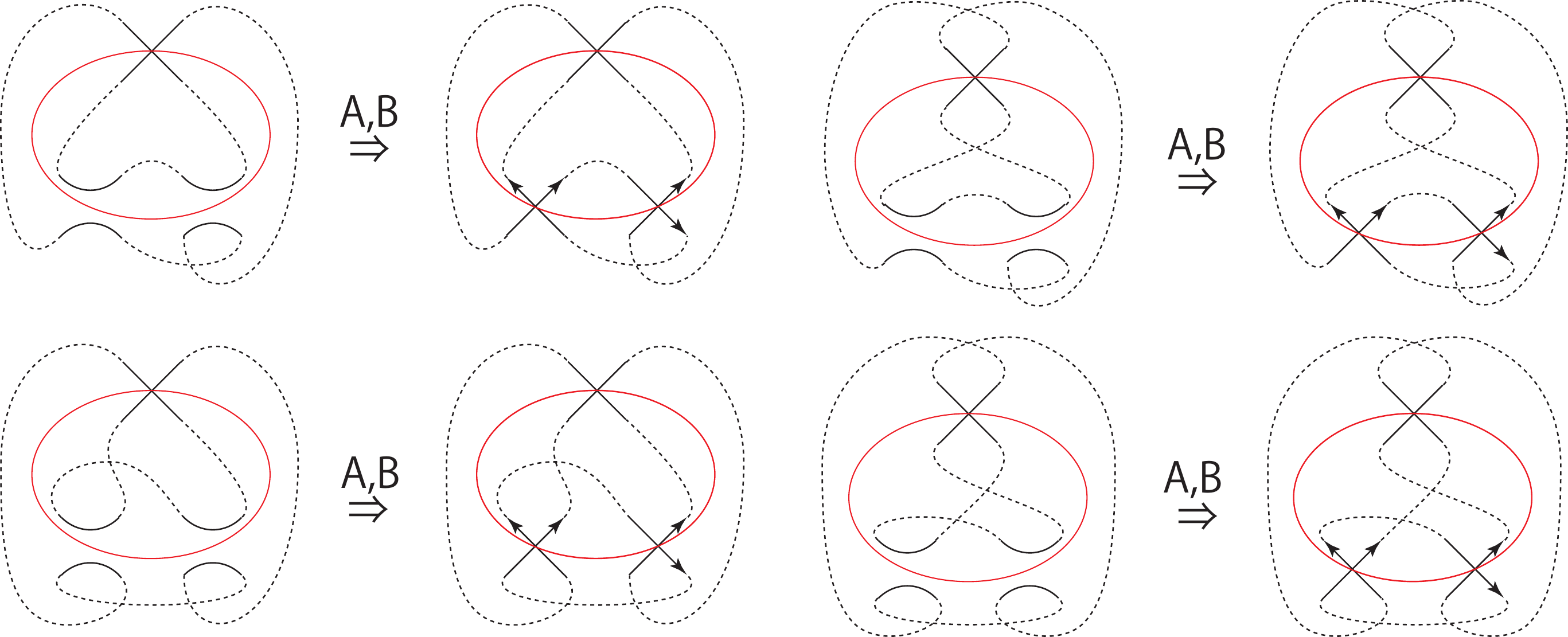}
\caption{Four cases obtaining knot projections with $t(P) \le 2$.}\label{rf14}
\end{figure}
This leads to Theorem~\ref{t_kiyakudo2_thm}.
\begin{theorem}\label{t_kiyakudo2_thm}
Let $P$ be a reduced knot projection.  
There is no possibility of existence of a simple closed curve, as shown in Fig.~\ref{rf6}, and there exists a circle intersecting $P$ at just two or just three double points of $P$, as shown in Fig.~\ref{taniyama}, 
if and only if $t(P)=2$.  
\begin{figure}[h!]
\includegraphics[width=8cm]{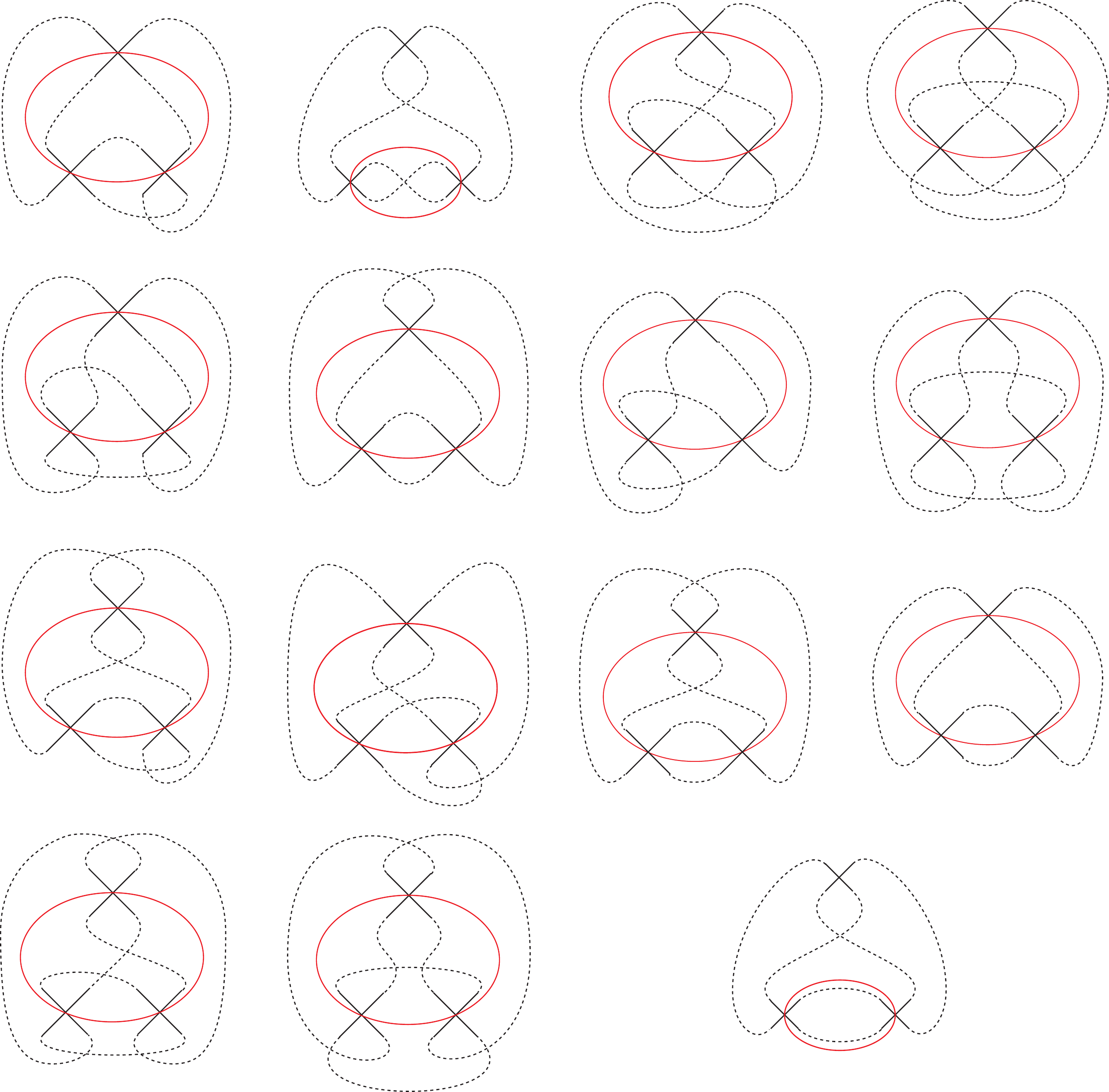}
\caption{Knot projections with $t(P) \le 2$.  
}\label{taniyama}
\end{figure}
\end{theorem}

\section{Tables}\label{sec_table}
In this section, we provide tables of prime reduced knot projections with the four types of reductivities.  A {\it{prime knot projection}} is defined as a knot projection that can not be a connected sum of two non-trivial knot projections.  A knot projection is called a {\it{prime reduced}} knot projection if the knot projection is prime and reduced (cf.~Sec.~\ref{intro}).  
\begin{table}[htbp]
\includegraphics[width=11cm]{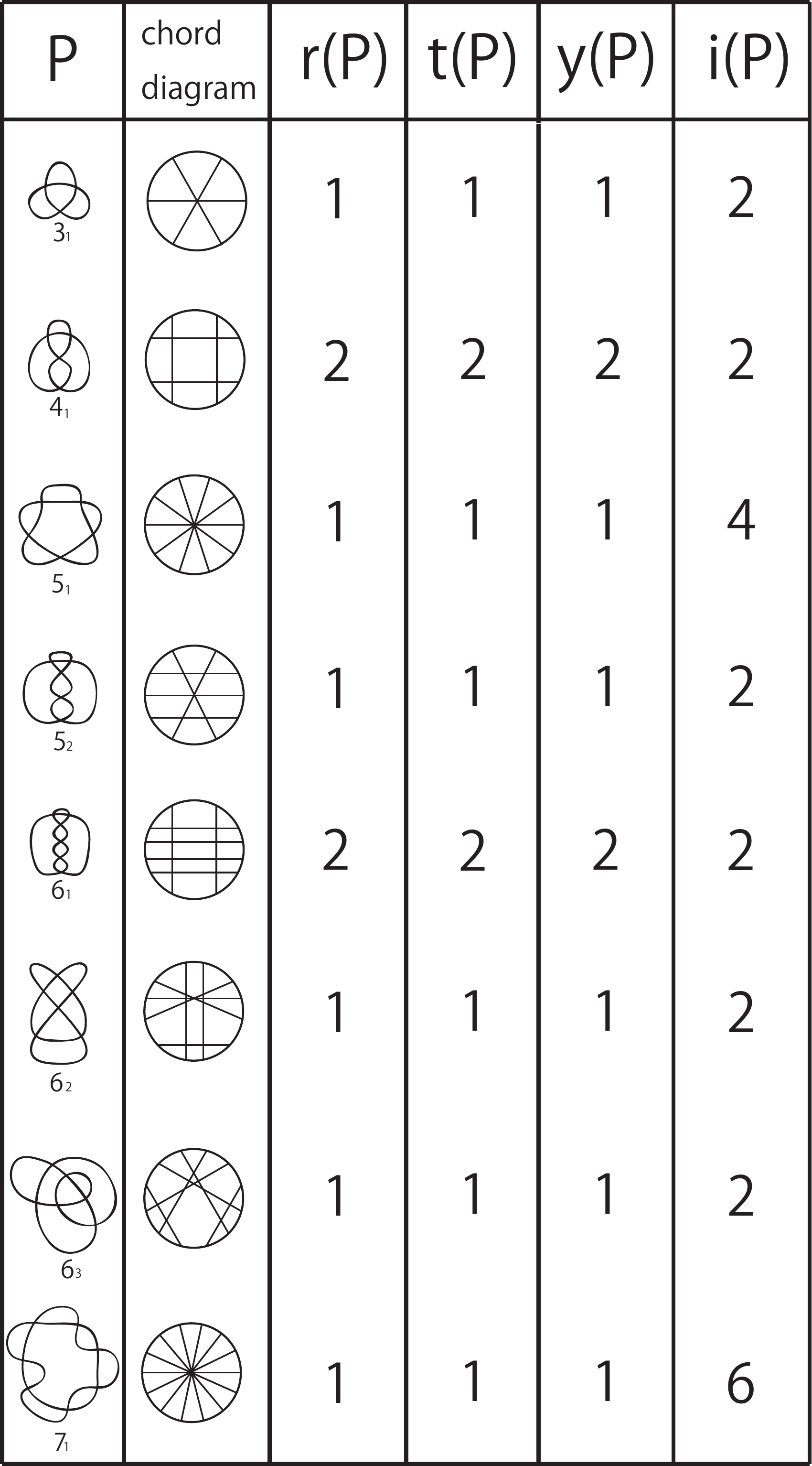}
\caption{Prime reduced knot projections and reductivities ($3_1$--$7_1$).}
\end{table}
\begin{table}[htbp]
\includegraphics[width=11cm]{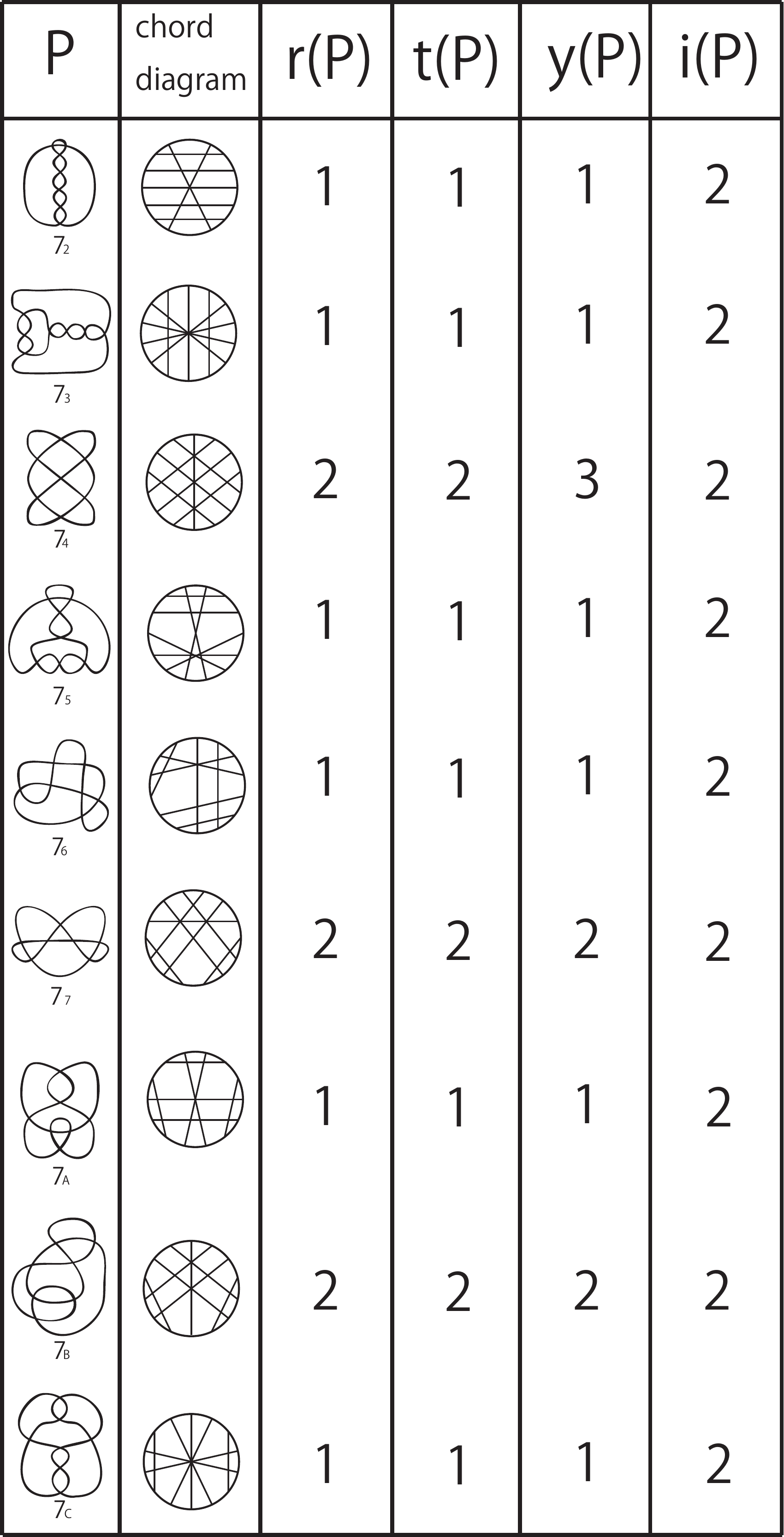}
\caption{Prime reduced knot projections and reductivities ($7_2$--$7_C$).}
\end{table}

\section*{Acknowledgements}
The authors would like to thank Professor Kouki Taniyama for his useful comments.  The authors also thank Ms.~Senja Barthel for fruitful discussion and providing us with her examples.  The part of this work contributed by N.~Ito was supported by a Waseda University Grant for Special Research Projects (Project number: 2014K-6292) and the JSPS Japanese-German Graduate Externship.


\begin{thebibliography}{99}
\bibitem{ItSh} N.~Ito and A.~Shimizu, The half-twisted operation on reduced knot projections, {\emph{J. Knot Theory Ramifications}} {\bf{21}} (2012), 1250112, 10pp.
\bibitem{ITcircle} N.~Ito and Y.~Takimura, Strong and weak (1, 2) homotopies on knot projections and new invariants, to appear in {\emph{Kobe J. Math}}.
\bibitem{ITtriple} N.~Ito and Y.~Takimura, Triple chords and strong (1, 2) homotopy, to appear in {\emph{J. Math.\ Soc.\ Japan}}.  
\bibitem{S} A.~Shimizu, The reductivity on spherical curves, to appear in {\emph{Topology Appl}}.  
\end{thebibliography}
\end{document}